\documentclass[10pt,reqno]{amsart}
\usepackage{amsmath,amsthm,amssymb,amsfonts,amscd}
\usepackage{mathrsfs}
\usepackage{bbm}
\usepackage{bbding}

\usepackage{bm}
\usepackage[backref=page]{hyperref}

\usepackage{geometry}
\usepackage{color}
\usepackage{xcolor}
\usepackage{enumerate}
\usepackage{picture,epic}
\usepackage{tikz}

\numberwithin{equation}{section}

\theoremstyle{plain}
\newtheorem{theorem}{Theorem}[section]
\newtheorem{lemma}[theorem]{Lemma}

\theoremstyle{definition}

\theoremstyle{remark}
\newtheorem{remark}[theorem]{Remark}

\renewcommand{\Re}{\operatorname{Re}}

\newcommand{\sym}{\operatorname{sym}}

\newcommand{\GL}{\operatorname{GL}}

\newcommand{\dd}{\mathrm{d}}

\newcommand{\Tam}{\operatorname{Tam}}

   \DeclareFontFamily{U}{wncy}{}
    \DeclareFontShape{U}{wncy}{m}{n}{<->wncyr10}{}
    \DeclareSymbolFont{mcy}{U}{wncy}{m}{n}
    \DeclareMathSymbol{\Sh}{\mathord}{mcy}{"58}

\makeatletter
\def\@tocline#1#2#3#4#5#6#7{\relax
  \ifnum #1>\c@tocdepth 
  \else
    \par \addpenalty\@secpenalty\addvspace{#2}%
    \begingroup \hyphenpenalty\@M
    \@ifempty{#4}{%
      \@tempdima\csname r@tocindent\number#1\endcsname\relax
    }{%
      \@tempdima#4\relax
    }%
    \parindent\z@ \leftskip#3\relax \advance\leftskip\@tempdima\relax
    \rightskip\@pnumwidth plus4em \parfillskip-\@pnumwidth
    #5\leavevmode\hskip-\@tempdima
      \ifcase #1
       \or\or \hskip 1em \or \hskip 2em \else \hskip 3em \fi%
      #6\nobreak\relax
    \hfill\hbox to\@pnumwidth{\@tocpagenum{#7}}\par
    \nobreak
    \endgroup
  \fi}
\makeatother

\begin{document}

\title[Extreme central $L$-values of elliptic curves]{Extreme central $L$-values of almost prime quadratic twists of  elliptic curves}

\author{Shenghao Hua and Bingrong Huang}

\address{Data Science Institute and School of Mathematics \\ Shandong University \\ Jinan \\ Shandong 250100 \\China}
\email{huashenghao@vip.qq.com}
\email{brhuang@sdu.edu.cn}

\date{\today}

\begin{abstract}
  In this paper, we prove the extreme values of $L$-functions at the central point for almost prime quadratic twists of an elliptic curve.
  As an application, we get the extreme values for the Tate--Shafarevich groups in the quadratic twist family of an elliptic curve under the Birch--Swinnerton-Dyer conjecture.
\end{abstract}


\dedicatory{On the 50th anniversary of Chen's theorem}


\thanks{This work was supported by  the National Key R\&D Program of China (No. 2021YFA1000700) and  NSFC (Nos. 12001314 and 12031008)} 

\maketitle


\section{Introduction} \label{sec:Intr}

The study of central $L$-values associated with quadratic characters is an active topic.
%
%
%
The extreme values of the quadratic Dirichlet $L$-functions at the central point was studied by Heath-Brown (unpublished, see \cite{HL}) and Soundararajan \cite{SD08}. Hoffstein--Lockhart \cite{HL} extended Heath-Brown's idea to prove a result of  quadratic twists of any modular form.
In this paper, we establish  a result of extreme central $L$-value for almost prime quadratic twists of an elliptic curve $E$, 
motivated by the application to the extreme values for the orders of the Tate--Shafarevich groups in the quadratic twist family of an elliptic curve, under the  Birch--Swinnerton-Dyer conjecture. For the proofs, we use the methods in Soundararajan \cite{SD08} and Radziwi{\l\l}--Soundararajan \cite{RS15}.

Let $E$ be an elliptic curve defined over $\mathbb{Q}$ with conductor $N$, and the associated normalized Hasse--Weil $L$-function is
\begin{equation}\label{lfun:HasseWeil}
  L(s,E)=\sum_{n=1}^{\infty}a(n)n^{-s}, \quad \Re(s)>1,
\end{equation}
where ``normalized" means the center of the critical strip is $s=1/2$, so its coefficients satisfy $|a(n)|\leq \tau(n)$ for all $n$. Here $\tau(\cdot)$ is the divisor function.
We have the completed $L$-function
\begin{equation}\label{lfun:qcomp}
  \Lambda(s,E)=\left(\frac{\sqrt{N}}{2\pi}\right)^s
  \Gamma\left(s+\frac{1}{2}\right)L(s,E),
\end{equation}
which has an analytic continuation to the entire complex plane and satisfies the functional equation
\begin{equation}\label{lfun:qfuneq}
  \Lambda(s,E)=\epsilon_E\Lambda(1-s,E)
\end{equation}
with the root number $\epsilon_E=\pm 1$.

Let $d$ be a fundamental discriminant with $(d,2N)=1$, and $\chi_d=(\frac{d}{\cdot})$ the associated primitive quadratic character.
We have
the quadratic twist of the elliptic curve $E$ by $\chi_d(\cdot)$, denoted by $E_d$.
The twisted $L$-function is
\begin{equation}\label{lfun:q}
  L(s,E_d)=\sum_{n=1}^{\infty}
  a(n)\chi_d(n)n^{-s},
\end{equation}
and the conductor of $E_d$ is $Nd^2$.
The associated completed $L$-function is
\begin{equation}\label{lfun:qcomp}
  \Lambda(s,E_d)=\left(\frac{\sqrt{N}|d|}{2\pi}\right)^s
  \Gamma\left(s+\frac{1}{2}\right)L(s,E_d),
\end{equation}
which also has an analytic continuation to the entire complex plane, and satisfies the following functional equation
\begin{equation}\label{lfun:qfuneq}
  \Lambda(s,E_d)=\epsilon_E(d)\Lambda(1-s,E_d),
\end{equation}
 with the root number
\begin{equation}\label{epsilon}
  \epsilon_E(d)=\epsilon_E\chi_d(-N).
\end{equation}

Waldspurger's theorem tells us that for all fundamental discriminant $d$ we have $L(\frac{1}{2},E_d)\geq 0$, and $L(\frac{1}{2},E_d)= 0$ when $\epsilon_E(d)=-1$.
We  should restrict attention to those quadratic twists which have root number 1.
Let
\begin{equation*}
  \Omega := \{d:\textrm{fundamental discriminant }d \textrm{ with } (d,2N)=1, \textrm{ and } \epsilon_E(d)=1\}.
\end{equation*}
Let $N_0=[8,N]$.
Let $\epsilon\in\{\pm 1\}$ and $a\pmod{N_0}$ denote a residue class with $a\equiv 1\textrm{ or }5\pmod 8$.
We assume that $\epsilon$ and $a$ are such that for any fundamental discriminant $d$ of sign $\epsilon$ with $d\equiv a \pmod{N_0}$, the root number $\epsilon_E(d)=1$. Define
\begin{equation*}
\Omega(a,\epsilon) := \{d\in \Omega : \  d\equiv a\pmod {N_0}, \ \epsilon d>0\} .
\end{equation*}
For a fixed $a$, the local $L$-function $L_{N_0}(s,E_d)$, which defined as
\begin{equation}
   L_{N_0}(s,E_d)=\sum_{\substack{n=1\\p\mid n\Rightarrow p\mid N_0}}^{\infty}
   \frac{a(n)}{n^s}\chi_d(n)
\end{equation}
for $\Re(s)>0$, should be a constant function of $d$ when $d\in \Omega(a,\pm 1)$, and we denote its value as $L_{a}(s)$.
In this paper, our main result is the following theorem.

\begin{theorem}\label{theorem:extremalvalue}
With the notions as above. For any  fixed  integer $W\geq 20$, we have
\begin{equation}
  \max_{\substack{d\in \Omega(a,\epsilon)
  \\ \frac{1}{2}X\leq |d| \leq \frac{5}{2}X \\ \omega(d)\leq W}}
  L\left(\frac{1}{2},E_d\right)
  \geq
\exp\left(\Big(2\sqrt{\frac{W-19.73}{22W+12}}
  +o(1)\Big)\frac{\sqrt{\log X}}{\sqrt{\log\log X}}\right),
\end{equation}
and
\begin{equation}
  \min_{\substack{d\in \Omega(a,\epsilon)
  \\ \frac{1}{2}X\leq |d| \leq \frac{5}{2}X \\ \omega(d)\leq W}}
  L\left(\frac{1}{2},E_d\right)
  \leq
\exp\left(-\Big(2\sqrt{\frac{W-19.73}{22W+12}}
  +o(1)\Big)\frac{\sqrt{\log X}}{\sqrt{\log\log X}}\right).
\end{equation}
Here $\omega(d)$ is the number of distinct prime divisors of $d$.
\end{theorem}

\begin{remark}
  One may use more complicated sieve methods and metaplectic techniques  to allow smaller value of $W$.
  Hoffstein--Luo \cite{HLuo} proved the nonvanishing result for this family with $d$ having no more than 4 prime factors.
  In the Goldbach problem, Chen \cite{chen} proved that there exists infinitely many prime numbers $p$ such that the number of prime divisors of $p+2$ is at most 2.
  One may also extend our results to the $\GL(3)$ case. See e.g. Hua--Huang \cite{HH22}.
\end{remark}

Our motivation of this work is to understand the extreme large values of  $|\Sh(E_d)|$, the order of Tate-Shafarevich group of $E_d$, under the Birch and Swinnerton-Dyer conjecture.
For $d\in\Omega$, let
\begin{equation*}
  S(E_d):=L\left(\frac{1}{2},E_d\right)
  \frac{|E_d(\mathbb{Q})_{tors}|^2}
  {\Omega(E_d) \Tam(E_d)},
\end{equation*}
where $|E_d(\mathbb{Q})_{tors}|$ denotes the size of the rational torsion group of $E_d$, $\Omega(E_d)$ is the real period of a minimal model for $E_d$,
and $\Tam(E_d)=\prod_{p}T_p(d)$ is the product of the Tamagawa numbers.
We have $T_p(d)=1$ for $p\nmid dN_0$, $T_p(d)$ is fixed for $p\mid N_0$ even if $d$ changes, and $T_p(d)\in\{1,2,4\}$ for $p\mid d$.
If $L(\frac{1}{2},E_d)\neq 0$, then the Birch and Swinnerton-Dyer conjecture predicts that $S(E_d)=|\Sh(E_d)|$.
It is natural to prove the extreme large value of $S(E_d)$ by using $L(1/2,E_d)$,
but $\Tam(E_d)$ may be large. Indeed, $\Tam(E_d)$ may have the same order of magnitude as $\tau(d)^2$, and from the Theorem 5.4 of \cite{Te15} we know $\tau(d)$ has the extremal large order as $\exp(\log 2\frac{\log d}{\log \log d})$.
To obtain the extreme large values of the order of the Tate-Shafarevich group, we may only consider $d$'s with $\omega(d)\ll 1$, so that  the effect of Tamagawa numbers in $S(E_d)$ is negligible.
We have the following result.

\begin{theorem}\label{theorem:tatesafgroup}
With the notions as above. For any fixed $W\geq20$, we have
  \begin{equation}
    \max_{\substack{d\in \Omega(a,\epsilon)
  \\ \frac{1}{2}X\leq |d| \leq \frac{5}{2}X \\ \omega(d)\leq W}}
  S(E_d)
  \geq \sqrt{X}
  \exp\left(\Big(2\sqrt{\frac{W-19.73}{22W+12}}
  +o(1)\Big)\frac{\sqrt{\log X}}{\sqrt{\log\log X}}\right),
  \end{equation}
as $X\to \infty$.
If the Birch and Swinnerton-Dyer conjecture for elliptic curves with analytic rank zero holds, then we have same results for  $|\Sh(E_d)|$.
\end{theorem}

\begin{proof}[Proof Theorem \ref{theorem:tatesafgroup} under Theorem \ref{theorem:extremalvalue}]
  Mazur \cite{Ma77} showed that $1\leq |E_d(\mathbb{Q})_{tors}|\leq 12$.
  The real period $\Omega(E_d)$ satisfies
\begin{equation*}
 \Omega(E_d)=\frac{\tilde{u}(d)}{\sqrt{|d|}}
 \Omega(E)
\end{equation*}
with some $\tilde{u}(d)\in \frac{1}{2}\mathbb{Z}$ decided by $E$ and $d$, and from the Proposition 2.5 of \cite{Pa12} we know that, for square-free $(d,2N)=1$, $\tilde{u}(d)=\tilde{u}(1)$ is a constant decided by $E$.
When $\omega(d)\leq W$ for $W$ fixed, we have $1\leq \Tam(E_d)=\prod_{p}T_p(d)\ll_{E,W} 1$.
By Theorem \ref{theorem:extremalvalue}, we complete the proof.
\end{proof}

%
%

The rest of this paper is organized as follows.
In \S \ref{sec:preliminaries}, we introduce some notation and present some lemmas that we will need later.
In \S \ref{sec:proof_EV}, we combine the resonance method in  Soundararajan \cite{SD08}, results in Radziwi{\l \l} and Soundararajan \cite{RS15},
and the sieve method to prove Theorem \ref{theorem:extremalvalue}.

\section{Notation and preliminary results}\label{sec:preliminaries}
Let $\Phi$ be a smooth non-negative Schwartz class function supported on $[\frac{1}{2},\frac{5}{2}]$ with $\Phi(x)=1$ for $x\in [1,2]$, and for any complex number $s$ let
\begin{equation}\label{def:checkPhi}
  \check\Phi(s)=\int_{0}^{\infty}\Phi(x)x^s\dd x.
\end{equation}
Let $L(s,\sym^2 E)$ be the symmetric square $L$-function of $E$ and  $L_p(s,\sym^2 E)$ be its local Euler factor at $p$. 

\begin{lemma}\cite[Proposition 1]{RS15} \label{lemma:char-moment}
Let $n$ and $\ell$ be positive integers with $(n,\ell)=1$, $(n\ell, N_0)=1$ and $\ell$ square-free.
Suppose that $\ell\sqrt{n}\leq X^{\frac{1}{2}-\varepsilon}$.
If $n$ is a square then
\begin{equation}\label{eqn:char-moment-square}
  \sum_{\substack{d\in \Omega(a,\epsilon)\\\ell\mid d}}
  \chi_d(n)\Phi\left(\frac{\epsilon d}{X}\right)
  =\check{\Phi}(0)\frac{X}{\ell N_0}
  \prod_{p\mid n\ell}\left(1+\frac{1}{p}\right)^{-1}
  \prod_{p\nmid N_0}\left(1-\frac{1}{p^2}\right)
  +O\left(X^{\frac{1}{2}+\varepsilon}\sqrt{n}\right).
\end{equation}
If $n$ is not a perfect square, then
\begin{equation}\label{eqn:char-moment-notsquare}
  \sum_{\substack{d\in \Omega(a,\epsilon)\\\ell\mid d}}
  \chi_d(n)\Phi\left(\frac{\epsilon d}{X}\right)
  =O\left(X^{\frac{1}{2}+\varepsilon}\sqrt{n}\right).
\end{equation}
\end{lemma}

\begin{lemma}\cite[Proposition 2]{RS15} \label{lemma:twisted-1-st-moment}
Let $u$ and $\ell$ be positive integers with $(u,\ell)=1$, $(u\ell, N_0)=1$ and $\ell$ square-free.
Define
\begin{equation}\label{eqn:twisted-1-st-moment}
  S(X;u,\ell)=\sum_{\substack{d\in \Omega(a,\epsilon)\\\ell\mid d}}
  L\left(\frac{1}{2},E_d\right) \chi_d(u)
  \Phi\left(\frac{\epsilon d}{X}\right).
\end{equation}
Write $u=u_1u_2^2$ with $u_1$ square-free.
Then for any $\varepsilon>0$,
\begin{equation*}
  S(X;u,\ell)=\frac{2Xa(u_1)}{\ell
  u_1^{\frac{1}{2}}N_0}
  \check{\Phi}(0)L_a\left(\frac{1}{2}\right)
  L(1,\sym^2 E)G(1;u,\ell)
  + O \left(X^{\frac{7}{8}+\varepsilon}
  u^{\frac{3}{8}}\ell^{\frac{1}{4}} \right),
\end{equation*}
where $G(1;u,\ell)=\prod_{p\textrm{ prime}}G_p(1;u,\ell)$ and
  \begin{equation}\label{def:Gpuv}
    G_p(1;u,\ell)=\left\{\begin{aligned}
    &L_p(1,\sym^2 E)^{-1}, &p\mid N_0,\\
    &\left(1-\frac{1}{p}\right)^2, &p\mid u_1,\\
    &\left(1-\frac{1}{p}\right) \left(1-\frac{1}{p^2}\right), &p\mid u\ell \textrm{ but }p\nmid u_1, \\
    &\left(1-\frac{1}{p}\right)^2
         \left(1+\frac{1}{p}
         \Big(1-\frac{\alpha_p^2}{p}\Big)
         \Big(1-\frac{\beta_p^2}{p}\Big)+\frac{1}{p}\right),
         &p\nmid u\ell N_0.
  \end{aligned}
  \right.
    \end{equation}
Here we write $a(p)=\alpha_p+\beta_p$ with $\alpha_p\beta_p=1$ and $|\alpha_p|=|\beta_p|=1$ for $p\nmid N$.
\end{lemma}

We can write $G(1;u,\ell)$ as $C(E)\tilde{h}(u)h(\ell)$ where $C(E)$ is a constant and $\tilde{h}$ and $h$ are multiplicative functions with $\tilde{h}(p^k)=1+O(1/p)$, $h(p)=1+O(1/p)$, and  $0<h(p),\tilde{h}(p),\tilde{h}(p^2)\leq 1$. More precisely,
let $C(E)=\prod_{p\textrm{ prime}}C_p(E)$ where
  \begin{equation}\label{def:CpE}
    C_p(E)=\left\{\begin{aligned}
    &L_p(1,\sym^2 E)^{-1}, &p\mid N_0,\\
         &\left(1-\frac{1}{p}\right)^2
         \left(1+\frac{1}{p}
         \Big(1-\frac{\alpha_p^2}{p}\Big)
         \Big(1-\frac{\beta_p^2}{p}\Big)+\frac{1}{p}\right),
         &p\nmid N_0,
  \end{aligned}
  \right.
    \end{equation}
$\tilde{h}(u)=\prod_{p^k \parallel u}\tilde{h}(p^k)$ with
    \begin{equation}\label{def:gpu}
    \tilde{h}(p^k)=    \left\{\begin{aligned}
    &\left(1+\frac{1}{p}
         \Big(1-\frac{\alpha_p^2}{p}\Big)
         \Big(1-\frac{\beta_p^2}{p}\Big)+\frac{1}{p}\right)^{-1}, & \textrm{if $k$ is odd},\\
        &\left(1+\frac{1}{p}\right)
        \left(1+\frac{1}{p}
         \Big(1-\frac{\alpha_p^2}{p}\Big)
         \Big(1-\frac{\beta_p^2}{p}\Big)+\frac{1}{p}\right)^{-1}, & \textrm{if $k$ is even},
    \end{aligned}
  \right.
    \end{equation}
and $h(\ell)=\prod_{p\mid \ell}h(p)$ with
  \begin{equation}\label{def:hpv}
h(p)= \left(1+\frac{1}{p}\right)
        \left(1+\frac{1}{p}
         \Big(1-\frac{\alpha_p^2}{p}\Big)
         \Big(1-\frac{\beta_p^2}{p}\Big)+\frac{1}{p}\right)^{-1}.
\end{equation}
Note that we have $L_a(\frac{1}{2})L(1,\sym^2 E)C(E)>0$.

We will need a strong version of the prime number theorem of an elliptic curve $L$-function.

\begin{lemma}\cite[Corollary 1.2]{LWY05}\label{lemma:sumofap2}
For an elliptic curve $E$ as before,
we have
\begin{equation*}
\sum_{p\leq x} a(p)^2 \log p = x + O_{E,c}(x \exp(-c \sqrt{\log x})).
\end{equation*}
\end{lemma}

Let $\mathcal{A}=(a_n)$ be a sequence of non-negative reals, and $\mathcal{P}$ a fixed set of primes.
Let $\ell$ be a square-free number, and write
$|\mathcal{A}_\ell|=\sum_{\ell\mid n}a_n$, and a non-negative multiplicative function $g(\ell)$ satisfying
\[0\leq g(p)<1. \]
Let $\mathcal{X}$ be a handy approximation to $|\mathcal A|$.
Let $r_\ell$ be defined by
\[
  |\mathcal A_\ell| = g(\ell) \mathcal X + r_\ell.
\]
Let 
\[
  S(\mathcal{A},z) =\sum_{\substack{(n,P(z))=1}}
  a_n
\]
where
\begin{equation*}
  P(z)=\prod_{\substack{p\in \mathcal{P}\\p<z}}p.
\end{equation*}
We will use the following lower bound of the linear  sieve.
\begin{lemma}\cite[Theorem 11.13]{FI10}\label{lemma:sieve}
With the notions as above.
Assume that $g(\ell)$ satisfies
\begin{equation}
  \prod_{\substack{w\leq p<z\\p\in \mathcal{P}}}(1-g(p))^{-1}\leq
  \left(\frac{\log z}{\log w}\right) \left(1+\frac{L}{\log w}\right)
\end{equation}
for every $2\leq w<z$, where $L\geq 1$ is constant.
Then for $s=\frac{\log D}{\log z}\geq 2$, we have
  \begin{equation}
  S(\mathcal{A},z) 
  \leq \mathcal XV(z)\left(F(s)+O((\log D)^{-\frac{1}{6}})\right)+R(D,z),
\end{equation}
\begin{equation}
  S(\mathcal{A},z) 
  \geq \mathcal X V(z)\left(f(s)+O((\log D)^{-\frac{1}{6}})\right)-R(D,z),
\end{equation}
where $V(z)=\prod_{\substack{p\in \mathcal{P}\\p\leq z}}(1-g(p))$, and
$F(s),f(s)$ are defined by the following system of differential-difference equations
\begin{equation}
  \left\{\begin{aligned}
    &sF(s)=2e^\gamma, &1\leq s\leq 3,\\
    &sf(s)=0, &s=2,\\
  \end{aligned}
  \right.
\end{equation}
\begin{equation}
  \left\{\begin{aligned}
    &(sF(s))^{'}=f(s-1), &s> 3,\\
    &(sf(s))^{'}=F(s-1), &s>2,\\
  \end{aligned}
  \right.
\end{equation}
$\gamma$ is the Euler constant, and
\begin{equation*}
  R(D,z) = \sum_{\substack{\ell\mid P(z)\\ \ell\leq D}}|r_\ell|.
\end{equation*}
In particular,  we have $f(s)> 0$ if $s>2$ and $F(s)=O(1)$ if $s>1$.
\end{lemma}


\section{Proof of Theorem \ref{theorem:extremalvalue}} \label{sec:proof_EV}

Let $M\leq X^{\frac{1}{22}-\varepsilon}$ be large. We set $\mathcal{L}=\sqrt{\log M\log\log M}$ and choose the resonator coefficients $b^{\pm}(m)$ to be a multiplicative function with $b^{\pm}(p)=\frac{\pm a(p)\mathcal{L}}{\sqrt{p}\log p}$ for $\mathcal{L}^2\leq p \leq \exp((\log \mathcal{L})^2)$, and $b^{\pm}(p^a)=0$ for all other primes when $a=1$ and all primes when $a\geq 2$. Let
\begin{equation}\label{def:resonator}
  R^{\pm}(d) :=\sum_{\textrm{odd }m\leq M}b^{\pm}(m)\chi_d(m).
\end{equation}
Define the  congruence sums
\begin{equation}\label{congruencesums}
\begin{split}
  C_{\ell,a,\epsilon}^{\pm}(X)
  &:=\sum_{\substack{d\in \Omega(a,\epsilon)\\ \ell\mid d}}L\left(\frac{1}{2},E_d\right)R^{\pm}(d)^2 \
  \Phi\left(\frac{\epsilon d}{X}\right).
\end{split}
\end{equation}

\begin{lemma}\label{lemma:C}
  With the notations as above. We have
  \begin{equation*}
    C_{\ell,a,\epsilon}^{\pm}(X)
    =g_1^{\pm}(\ell)
     \mathcal{X}_1
     +r_{\ell,a,\epsilon}^{\pm}(X),
  \end{equation*}
  where
  \[
    \mathcal{X}_1 =
  \prod_{p\nmid N_0}\Big(1+b^{\pm}(p)^2\tilde{h}(p^2)
+2a(p)b^{\pm}(p)\tilde{h}(p)p^{-\frac{1}{2}}
\Big)
    \frac{2X}{N_0}\check{\Phi}(0)
L_a\left(\frac{1}{2}\right)L(1,\sym^2 E)C(E),
  \]
  $g_1^\pm(\ell)$ are multiplicative functions  supported on square-free numbers with
\begin{equation} \label{eqn:g1}
  g_1^{\pm}(p)=\frac{h(p)}{p}\Big(1
  +b^{\pm}(p)^2\tilde{h}(p^2)
+2a(p)b^{\pm}(p)\tilde{h}(p)p^{-\frac{1}{2}}
\Big)^{-1}
\end{equation}
  when $(p,N_0)=1$, and $g_1^{\pm}(p)=0$ otherwise, and
  \begin{equation}\label{eqn:R1}
r_{\ell,a,\epsilon}^{\pm}(X)
\ll g_1^{\pm}(\ell)
    \mathcal{X}_1 \exp\left(-\frac{\log M}{(\log\log M)^4}\right)
  + X^{\frac{7}{8}+\varepsilon}
  M^{\frac{11}{4}}\ell^{\frac{1}{4}}.
\end{equation}
\end{lemma}

\begin{proof}
Let $M$ and $X$ be large enough such that all prime factors of $N_0$ is strictly small than $\mathcal{L}=\sqrt{\log M\log\log M}$. 
By \eqref{def:resonator}, \eqref{congruencesums}, and Lemma \ref{lemma:twisted-1-st-moment},
we have
\begin{equation*}
\begin{split}
  C_{\ell,a,\epsilon}^{\pm}(X)
  &=\underset{\substack{m_1,m_2\leq M\\(m_1m_2,\ell N_0)=1}} {\sum\sum}b^{\pm}(m_1)b^{\pm}(m_2)
  \sum_{\substack{d\in \Omega(a,\epsilon)\\ \ell\mid d\\(d,m_1m_2)=1}}
  L\left(\frac{1}{2},E_d\right) \chi_{d}(m_1m_2)
  \Phi\left(\frac{\epsilon d}{X}\right) \\
  &=\frac{h(\ell)}{\ell}\frac{2X}{N_0}
  \check{\Phi}(0)L_a\left(\frac{1}{2}\right)L(1,\sym^2 E)C(E)\\
  &\hskip 30 pt\times
  \underset{\substack{m_1,m_2\leq M\\(m_1m_2,\ell N_0)=1}}
  {\sum\sum}b^{\pm}(m_1)b^{\pm}(m_2)\tilde{h}(m_1m_2)
  \frac{a(\frac{m_1m_2}{(m_1,m_2)^2})}
  {(\frac{m_1m_2}{(m_1,m_2)^2})^{\frac{1}{2}}}
  +O(X^{\frac{7}{8}+\varepsilon}
  M^{\frac{11}{4}}\ell^{\frac{1}{4}}).
 \end{split}
\end{equation*}
 Here we have used the fact  $\chi_d(m)=0$ when $(m,d)\neq 1$.
Note that  $m_1,m_2$ in the above sums are square-free. We write $m_1=qr$, $m_2=qs$ with $(q,r)=(q,s)=(r,s)=1$, getting
\begin{multline}\label{eqn:sum_qrs1}
  \underset{\substack{m_1,m_2\leq M\\(m_1m_2,\ell N_0)=1}}
  {\sum\sum}b^{\pm}(m_1)b^{\pm}(m_2) \tilde{h}(m_1m_2)
  \frac{a(\frac{m_1m_2}{(m_1,m_2)^2})}
  {(\frac{m_1m_2}{(m_1,m_2)^2})^{\frac{1}{2}}} \\
  =  \underset{\substack{(q,r)=(q,s)=(r,s)=1
  \\qr,qs\leq M\\(qrs,\ell N_0)=1}}
  {\sum\sum\sum}b^{\pm}(q)^2\tilde{h}(q^2)a(r)
  b^{\pm}(r)\tilde{h}(r)
a(s)b^{\pm}(s)\tilde{h}(s)(rs)^{-\frac{1}{2}}.
  \end{multline}
The above sums  without the restrictions $qr,qs\leq M$ are
\begin{multline}\label{eqn:sum_qrs2}
  \underset{\substack{(q,r)=(q,s)=(r,s)=1
  \\(qrs,\ell N_0)=1}}
  {\sum\sum\sum}b^{\pm}(q)^2\tilde{h}(q^2)
  a(r)b^{\pm}(r)\tilde{h}(r)
a(s)b^{\pm}(s)\tilde{h}(s)(rs)^{-\frac{1}{2}}
  \\
  =\prod_{p\nmid \ell N_0}\Big(1+b^{\pm}(p)^2\tilde{h}(p^2)
+2a(p)b^{\pm}(p)\tilde{h}(p)p^{-\frac{1}{2}}
\Big).
\end{multline}
By using the symmetry of $qr$ and $qs$, the difference between \eqref{eqn:sum_qrs1} and \eqref{eqn:sum_qrs2} is bounded by
  \begin{equation*}
  2 \underset{\substack{(q,r)=(q,s)=(r,s)=1
  \\qr> M\\(qrs,\ell N_0)=1}}
  {\sum\sum\sum}|b^{\pm}(q)^2\tilde{h}(q^2)
  a(r)b^{\pm}(r)\tilde{h}(r)
  a(s)b^{\pm}(s)\tilde{h}(s)|
(rs)^{-\frac{1}{2}}.
\end{equation*}
By using Rankin's trick, for any $\alpha>0$, the above is
\begin{equation}\label{eqn:Rankinerror}
\begin{split}
  &\ll M^{-\alpha}\underset{\substack{(q,r)=(q,s)=(r,s)=1
  \\(qrs,\ell N_0)=1}}
  {\sum\sum\sum}|b^{\pm}(q)^2\tilde{h}(q^2)
  a(r)b^{\pm}(r)\tilde{h}(r)
  a(s)b^{\pm}(s)\tilde{h}(s)|
  q^{\alpha}r^{-\frac{1}{2}+\alpha}
s^{-\frac{1}{2}}\\
  &\ll M^{-\alpha}\prod_{p\nmid \ell N_0}\Big(1+b^{\pm}(p)^2\tilde{h}(p^2)p^{\alpha}
  \pm a(p)b^{\pm}(p)\tilde{h}(p)
  p^{-\frac{1}{2}}(1+p^{\alpha})\Big),
\end{split}
\end{equation}
because $|a(p)b^{\pm}(p)|=\pm a(p)b^{\pm}(p)$.
Taking $\alpha=(\log \mathcal{L})^{-3}$, we may see the ratio of the error term \eqref{eqn:Rankinerror} to the main term \eqref{eqn:sum_qrs2} is
\begin{equation}\label{eqn:ratioerror}
\begin{split}
  &\ll M^{-\alpha}\prod_{p\nmid \ell N_0}\Bigg(1+b^{\pm}(p)^2\tilde{h}(p^2)
  \Big(p^\alpha-1\Big)
  + a(p)b^{\pm}(p)\tilde{h}(p)
  p^{-\frac{1}{2}}\Big(\pm(1+p^{\alpha})-2\Big)
  \Bigg)\\
  &\ll \exp\Bigg(
  -\alpha\log M+\sum_{\mathcal{L}^2\leq p\leq \exp((\log \mathcal{L})^2)}
   \Bigg( \frac{a(p)^2 \mathcal{L}}{p\log p}\Big(\Big(\frac{\mathcal{L}}{\log p}+1\Big)\Big(p^\alpha-1\Big)+4\Big)\Bigg)
  \Bigg)\\
  & \ll \exp\Bigg( -\alpha\log M+
\sum_{\mathcal{L}^2\leq p\leq \exp((\log \mathcal{L})^2)}
   \Bigg( \frac{a(p)^2 \mathcal{L}}{p\log p}\Big(\frac{\mathcal{L}+\log p}{(\log \mathcal{L})^3}\Big(1+O\Big(\frac{\log p}{(\log \mathcal{L})^3}\Big)\Big)+4\Big)\Bigg)
  \Bigg) \\
  & \ll \exp\left(-\alpha\frac{\log M}{\log\log M}\right),
\end{split}
\end{equation}
with some calculation using the prime number theorem in Lemma \ref{lemma:sumofap2}. 
Thus
\begin{multline*}
   C_{\ell,a,\epsilon}^{\pm}(X)
   =\frac{h(\ell)}{\ell}
  \prod_{p\nmid \ell N_0}\Big(1+b^{\pm}(p)^2\tilde{h}(p^2)
+2a(p)b^{\pm}(p)\tilde{h}(p)p^{-\frac{1}{2}}
\Big)
    \frac{2X}{N_0}\check{\Phi}(0)
L_a\left(\frac{1}{2}\right)L(1,\sym^2 E)C(E)\\
\times
\left(1+O\Big(\exp(-\frac{\log M}{(\log\log M)^4})\Big)\right)
  +O\left(X^{\frac{7}{8}+\varepsilon}
  M^{\frac{11}{4}}\ell^{\frac{1}{4}}\right).
\end{multline*}
This completes the proof of the Lemma.
\end{proof}

\begin{lemma}\label{lemma:mean-value-1}
  Let $\mathcal{P}=\{p:p\textrm{ prime and }p\nmid N_0\}$.
  Let $X\geq 10$ be sufficiently large and $\varepsilon\in (0,1/1000)$ sufficiently small.
  Let $D\geq X^\varepsilon$ and $M\geq X^{\varepsilon}$ satisfy  that $D M^{\frac{11}{5}} \leq X^{\frac{1}{10}-\varepsilon}$.
  Then for $z \leq D^{1/2-\varepsilon}$, we have
  \begin{equation*}
   \sum_{\substack{d\in \Omega(a,\epsilon)\\(d,P(z))=1}}
  L\left(\frac{1}{2},E_d\right) R^{+}(d)^2
  \Phi\left(\frac{\epsilon d}{X}\right)
   \\
  \gg \frac{X}{\log z}\prod_{p\nmid N_0}\Big(1+b^{+}(p)^2\tilde{h}(p^2)
+2a(p)b^{+}(p)\tilde{h}(p)p^{-\frac{1}{2}}
\Big),
\end{equation*}
\begin{equation*}
   \sum_{\substack{d\in \Omega(a,\epsilon)\\(d,P(z))=1}}
  L\left(\frac{1}{2},E_d\right) R^{-}(d)^2
  \Phi\left(\frac{\epsilon d}{X}\right)
   \\
  \ll \frac{X}{\log z}\prod_{p\nmid N_0}\Big(1+b^{-}(p)^2\tilde{h}(p^2)
+2a(p)b^{-}(p)\tilde{h}(p)p^{-\frac{1}{2}}
\Big).
\end{equation*}
\end{lemma}

\begin{proof}
Recall $0<h(p), \tilde{h}(p), \tilde{h}(p^2)\leq 1$, $|a(p)|\leq 2$.
For any $z,w$ we have
\begin{equation}
\begin{split}
 \prod_{\substack{w\leq p<z\\p\in \mathcal{P}}}(1-g_1^{\pm}(p))^{-1}
  &=\prod_{\substack{w\leq p<z\\p\nmid N_0}} \Bigg(1+\frac{h(p)}{p+b^{\pm}(p)^2\tilde{h}(p^2)p
+2a(p)b^{\pm}(p)\tilde{h}(p)
p^{\frac{1}{2}} -h(p)}\Bigg)\\
&\leq \prod_{\substack{w\leq p<z\\p\nmid N_0\\ b^{\pm}(p)=0}} \Bigg(1+\frac{1}{p-1}\Bigg)
\prod_{\substack{w\leq p<z\\p\nmid N_0\\ b^{\pm}(p)\neq 0}} \Bigg(1+\frac{1}{p-\frac{4\mathcal{L}}{\log p}-\frac{4\mathcal{L}^2}{(\log p)^2}-1}\Bigg)\\
&\leq \prod_{\substack{w\leq p<z\\p\nmid N_0}} \Bigg(1+\frac{1}{p-1}\Bigg)
\prod_{\substack{w\leq p<z\\ b^{\pm}(p)\neq 0}} \Bigg(1+\frac{\frac{4\mathcal{L}}{\log p}+\frac{4\mathcal{L}^2}{(\log p)^2}}{p^2}\Big(1+O\Big(\frac{\mathcal{L}}{p\log p}+\frac{\mathcal{L}^2}{p(\log p)^2}\Big)\Big)\Bigg)\\
&\leq \prod_{w\leq p<z} \Bigg(1+\frac{1}{p-1}\Bigg)
\prod_{\substack{w\leq p<z\\b^{\pm}(p)\neq 0}} \Bigg(1+\frac{4}{p(\log p)^2}\Big(1+O\Big(\frac{1}{(\log p)^2}\Big)\Big)\Bigg)\\
&\leq \Big(\frac{\log z}{\log w}\Big)\Big(1+\frac{L_1}{\log w}\Big)
  \end{split}
\end{equation}
for some absolute constant $L_1\geq 1$.
For the third inequality, we have used the fact $p\geq \mathcal{L}^2$ when $b^{\pm}(p)\neq 0$.
We will apply Lemma \ref{lemma:sieve} with
\begin{equation}
  a_n=\left\{\begin{aligned}
  &L\Big(\frac{1}{2},E_{\epsilon n}\Big)  R^{\pm}(\epsilon n) ^2 \Phi\Big(\frac{n}{X}\Big), &\epsilon n=d \in \Omega(a,\epsilon),\\
  &0, & \textrm{otherwise},
    \end{aligned}
  \right.
\end{equation}
$g(\ell)=g_1^{\pm}(\ell)$, and $r_\ell = r_{\ell,a,\epsilon}^{\pm}(X)$ in Lemma \ref{lemma:C}.
Note that we have
\begin{equation*}
  \sum_{\substack{\ell\leq D\\ \mu(\ell)^2=1}}|r_{\ell,a,\epsilon}^{\pm}(X)|
  = O(\mathcal{X}_1 (\log X)^{-A})
\end{equation*}
for any large $A$, when $D M^{\frac{11}{5}} \leq X^{\frac{1}{10}-\varepsilon}$ and $M\geq X^{\varepsilon}$.
Assume $z \leq D^{1/2-\varepsilon}$. Then  $\frac{\log D}{\log z}>2$.
By using Lemma \ref{lemma:sieve}
with $P(z)=\prod_{\substack{p<z\\p\nmid N_0}}p$, we have
\begin{align*}
   \sum_{\substack{d\in \Omega(a,\epsilon)\\(d,P(z))=1}}
  L\left(\frac{1}{2},E_d\right) R^{+}(d)^2
  \Phi\left(\frac{\epsilon d}{X}\right)
  & \gg \mathcal{X}_1
  \prod_{p<z}(1-g_1^{+}(p))\\
  & \gg \frac{X}{\log z}\prod_{p\nmid N_0}\Big(1+b^{+}(p)^2\tilde{h}(p^2)
+2a(p)b^{+}(p)\tilde{h}(p)p^{-\frac{1}{2}}
\Big),
\end{align*}
and
\begin{align*}
   \sum_{\substack{d\in \Omega(a,\epsilon)\\(d,P(z))=1}}
  L\left(\frac{1}{2},E_d\right) R^{-}(d)^2
  \Phi\left(\frac{\epsilon d}{X}\right)
  & \ll \mathcal{X}_1
  \prod_{p<z}(1-g_1^{-}(p))\\
  & \ll \frac{X}{\log z}\prod_{p\nmid N_0}\Big(1+b^{-}(p)^2\tilde{h}(p^2)
+2a(p)b^{-}(p)\tilde{h}(p)p^{-\frac{1}{2}}
\Big).
\end{align*}
Here we have used the fact $\prod_{p<z}(1-g_1^{-}(p)) \asymp 1/\log z$.
This completes the proof of the Lemma.
\end{proof}

\begin{lemma}\label{lemma:mean-value-0}
  Let $\mathcal{P}=\{p:p\textrm{ prime and }p\nmid N_0\}$.
  Let $X\geq 10$ be sufficiently large and $\varepsilon\in (0,1/1000)$ sufficiently small.
  Let $D\geq X^\varepsilon$ and $M\geq X^{\varepsilon}$  satisfy $D M^{3} \leq X^{\frac{1}{2}-\varepsilon}$.
  Then for $ z \leq D^{1/2-\varepsilon}$, we have
  \begin{equation*}
   \sum_{\substack{d\in \Omega(a,\epsilon)\\(d,P(z))=1}}
  R^{+}(d)^2
  \Phi\left(\frac{\epsilon d}{X}\right)
  \ll \frac{X}{\log z}\prod_{p\nmid N_0}\Big(1+b^{+}(p)^2\frac{p}{p+1}\Big),
\end{equation*}
\begin{equation*}
   \sum_{\substack{d\in \Omega(a,\epsilon)\\(d,P(z))=1}}
  R^{-}(d)^2
  \Phi\left(\frac{\epsilon d}{X}\right)
  \gg \frac{X}{\log z}\prod_{p\nmid N_0}
  \Big(1+b^{-}(p)^2\frac{p}{p+1}\Big).
\end{equation*}
\end{lemma}

\begin{proof}
Define the congruence sums
\begin{equation}\label{eqn:D}
  D_{\ell,a,\epsilon}^{\pm}(X)
  :=\sum_{\substack{d\in \Omega(a,\epsilon)\\ \ell\mid d}}
  R^{\pm}(d)^2 \Phi\left(\frac{\epsilon d}{X}\right).
\end{equation}
By \eqref{def:resonator} and Lemma \ref{lemma:char-moment}, we have
\begin{equation}
\begin{split}
  D_{\ell,a,\epsilon}^{\pm}(X)
  &=\sum_{m_1,m_2\leq M}
  b^{\pm}(m_1)b^{\pm}(m_2) \sum_{\substack{d\in \Omega(a,\epsilon)\\ \ell\mid d}}
  \chi_d(m_1m_2)\Phi\left(\frac{\epsilon d}{X}\right)\\
  &=\prod_{p\mid \ell} \frac{1}{p+1}\frac{X}{N_0\zeta(2)}\check{\Phi}(0)
\prod_{p\mid N_0} \left(1-\frac{1}{p^2}\right)^{-1}
  \sum_{\substack{m\leq M\\(m,\ell N_0=1)}}b^{\pm}(m)^2
  \prod_{p\mid m}\frac{p}{p+1}
  +O\left(M^3X^{\frac{1}{2}+\varepsilon}\right).
  \end{split}
\end{equation}
Rankin's trick shows that for any $\alpha>0$,
we have
\begin{equation}
  \sum_{\substack{m\leq M\\(m,\ell N_0=1)}}b^{\pm}(m)^2
  \prod_{p\mid m}\frac{p}{p+1}
  =\prod_{p\nmid \ell N_0}
  \left(1+b^{\pm}(p)^2\frac{p}{p+1}\right)
  +O\left(M^{-\alpha}\prod_{p\nmid \ell N_0}
  (1+b^{\pm}(p)^2p^\alpha\frac{p}{p+1})\right).
\end{equation}
Take $\alpha=(\log \mathcal{L})^{-3}$. Then by using Lemma \ref{lemma:sumofap2}, we see the ratio of the above error term  to the main term is
\begin{equation}\label{eqn:ratioerror}
\begin{split}
  &\ll M^{-\alpha}\prod_{p\nmid \ell N_0}
  \Bigg(1+b^{\pm}(p)^2p^\alpha\frac{p}{p+1}
  \Big(p^{\alpha}-1\Big)\Bigg)
  \\
  &\ll \exp\Bigg(
  -\alpha\log M+\sum_{\mathcal{L}^2\leq p\leq \exp((\log \mathcal{L})^2)}\Big( \frac{a(p)^2\mathcal{L}^2}
  {p(\log p)^2}\Big(p^\alpha-1\Big)\Big)
  \Bigg)\\
  & \ll \exp\Bigg(
  -\alpha\log M+\frac{\mathcal{L}^2}{(\log \mathcal{L})^3}\sum_{\mathcal{L}^2\leq p\leq \exp((\log \mathcal{L})^2)}
  \frac{a(p)^2}{p\log p}\Big(1+O\Big(\frac{\log p}{(\log \mathcal{L})^3}\Big)\Big)
  \Bigg)\\
     & \ll \exp\left(-\alpha\frac{\log M}{\log\log M}\right).
\end{split}
\end{equation}
Then we have
\begin{equation}
\begin{split}
  D_{\ell,a,\epsilon}^{\pm}(X)
  &=\prod_{p\mid \ell} \frac{1}{p+1} \frac{X}{N_0\zeta(2)}\check{\Phi}(0)
\prod_{p\mid N_0} \left(1-\frac{1}{p^2}\right)^{-1}
\prod_{p\nmid \ell N_0}
  \left(1+b^{\pm}(p)^2\frac{p}{p+1}\right) \\
  &\hskip 60 pt  \times
  \left(1+O \Big(\exp\Big(-\frac{\log M}{(\log\log M)^4}\Big)\Big)\right)
  +O\left(M^3X^{\frac{1}{2}+\varepsilon}\right).
 \end{split}
\end{equation}
We define the multiplicative functions $g_2^{\pm}$ supported on square-free numbers as
\begin{equation*}
  g_2^{\pm}(p) = \frac{1}{p+1}
  \left(1+b^{\pm}(p)^2\frac{p}{p+1}\right)^{-1},
\end{equation*}
when $(p,N_0)=1$, and $g_2^{\pm}(p)=0$ otherwise.
Hence
\begin{equation*}
    D_{\ell,a,\epsilon}^{\pm}(X)
    =g_2^{\pm}(\ell) \mathcal{X}_2
     +r_{2,\ell}^{\pm}(X),
\end{equation*}
     where
\[
  \mathcal{X}_2 =
  \frac{X}{N_0\zeta(2)}\check{\Phi}(0)
\prod_{p\mid N_0} \left(1-\frac{1}{p^2}\right)^{-1}
\prod_{p\nmid  N_0}
  \left(1+b^{\pm}(p)^2\frac{p}{p+1}\right) ,
\]
and
\begin{equation*}
  r_{2,\ell}^{\pm}(X)
  = O\left( g_2^{\pm}(\ell)
      \mathcal{X}_2 \exp\Big(-\frac{\log M}{(\log\log M)^4}\Big)
  + M^3X^{\frac{1}{2}+\varepsilon}\right).
\end{equation*}

Let $\mathcal{P}=\{p:p\textrm{ prime and }p\nmid N_0\}$, recall $|a(p)|\leq 2$, then for any $z>w\geq2$ we have
\begin{equation}
\begin{split}
  \prod_{\substack{w\leq p<z\\p\in \mathcal{P}}}(1-g_2^{\pm}(p))^{-1}
  &=\prod_{\substack{w\leq p<z\\p\nmid N_0}} \Big(1+\frac{1}{p+b^{\pm}(p)^2p}\Big) \\
&\leq \prod_{\substack{w\leq p<z\\p\nmid N_0}} \Big(1+\frac{1}{p}\Big) \\
&\leq \left(\frac{\log z}{\log w}\right) \left(1+\frac{L_2}{\log w}\right)
  \end{split}
\end{equation}
for some absolute constant $L_2\geq 1$.
We will apply Lemma \ref{lemma:sieve} with
\begin{equation}
  a_n=\left\{\begin{aligned}
  &  R^{\pm}(\epsilon n) ^2 \Phi\left(\frac{n}{X}\right), &\epsilon n=d \in \Omega(a,\epsilon),\\
  &0, &otherwise,
    \end{aligned}
  \right.
\end{equation}
$g(\ell)=g_2^{\pm}(\ell)$, and $r_\ell=r_{2,\ell}^{\pm}(X)$.
We have
\begin{equation*}
  \sum_{\substack{\ell\leq D\\ \mu(\ell)^2=1}}|r_{2,\ell}^{\pm}(X)|
  = O\left( \mathcal{X}_2 (\log X)^{-A} \right)
\end{equation*}
for any large $A$, when $D\leq X^{\frac{1}{2}-\varepsilon}M^{-3}$ and $M\geq X^\varepsilon$.
Assume $z\leq D^{1/2-\varepsilon}$. Then $\log D/ \log z >2$.
By Lemma \ref{lemma:sieve} with $P(z)=\prod_{\substack{p<z\\p\nmid N_0}}p$, we have
\begin{equation}
   \sum_{\substack{d\in \Omega(a,\epsilon)\\(d,P(z))=1}}
  R^{+}(d)^2
  \Phi\left(\frac{\epsilon d}{X}\right)
  \ll \mathcal{X}_2
  \prod_{p<z}(1-g_2^{+}(p))
  \ll \frac{X}{\log z}\prod_{p\nmid N_0}\Big(1+b^{+}(p)^2\frac{p}{p+1}\Big),
\end{equation}
and
\begin{equation}
   \sum_{\substack{d\in \Omega(a,\epsilon)\\(d,P(z))=1}}
  R^{-}(d)^2
  \Phi\left(\frac{\epsilon d}{X}\right)
  \gg \mathcal{X}_2
  \prod_{p<z}(1-g_2^{-}(p))
  \gg \frac{X}{\log z}\prod_{p\nmid N_0}
  \Big(1+b^{-}(p)^2\frac{p}{p+1}\Big),
\end{equation}
as claimed. Here we have used the fact $\prod_{p<z}(1-g_2^{-}(p)) \asymp 1/\log z$.
\end{proof}

Now we are ready to prove Theorem \ref{theorem:extremalvalue}.

\begin{proof}[Proof of Theorem \ref{theorem:extremalvalue}]
Fix an integer $W\geq 20$. Let $X$ be large enough, $s=\frac{\log D}{\log z}= 2.023$,  $D= X^{\frac{2.023}{W+0.5}}\leq X^{\frac{1}{10}-\varepsilon}$, and $z=X^{\frac{1}{W+0.5}}$.
Then we have $[\frac{\log X}{\log z}]=W$.
Let
\begin{equation}
  M=X^{\frac{W-19.73}{22W+12}}
  \leq X^{\frac{1}{22}-\varepsilon}
  D^{-\frac{5}{11}}
  =X^{\frac{W-19.73}{22W+11}-\varepsilon}.
\end{equation}
Thus by Lemma \ref{lemma:mean-value-1}, Lemma \ref{lemma:mean-value-0}, and Lemma \ref{lemma:sumofap2}, we have
\begin{align*}
    \max_{ \substack{d\in \Omega(a,\epsilon)
  \\ \frac{1}{2}X\leq |d| \leq \frac{5}{2}X \\ \omega(d)\leq W}}
  L\left(\frac{1}{2},E_d\right)
   & \geq \max_{ \substack{d\in \Omega(a,\epsilon)\\ \frac{1}{2}X\leq |d| \leq \frac{5}{2}X \\(d,P(z))=1}}
  L\left(\frac{1}{2},E_d\right) \\
  &
  \geq \bigg(\sum_{\substack{d\in \Omega(a,\epsilon)
  \\(d,P(z))=1}}
  L\left(\frac{1}{2},E_d\right)R^{+}(d)^2
  \Phi\left(\frac{\epsilon d}{X}\right)\bigg) \bigg/
  \bigg(\sum_{\substack{d\in \Omega(a,\epsilon)\\(d,P(z))=1}}
  R^{+}(d)^2
  \Phi\left(\frac{\epsilon d}{X}\right)\bigg)
  \\
  & \gg \prod_{\substack{\mathcal{L}^2\leq p\leq \exp((\log \mathcal{L})^2)}}
  \left(1+2a(p)b^{+}(p)\tilde{h}(p)p^{-\frac{1}{2}}\right) \\
  & \gg \exp\Big(2\mathcal{L}\sum_{\mathcal{L}^2\leq p\leq \exp((\log \mathcal{L})^2)}
\frac{a(p)^2}{p\log p}\Big)\\
  & =\exp\left(\left(2\sqrt{\frac{W-19.73}{22W+12}}
  +o(1)\right)\frac{\sqrt{\log X}}{\sqrt{\log\log X}}\right),
\end{align*}
and
\begin{align*}
    \min_{  \substack{d\in \Omega(a,\epsilon)
  \\ \frac{1}{2}X\leq |d| \leq \frac{5}{2}X \\ \omega(d)\leq W}}
  L\left(\frac{1}{2},E_d\right)
  & \leq \min_{  \substack{d\in \Omega(a,\epsilon)\\ \frac{1}{2}X\leq |d| \leq \frac{5}{2}X \\(d,P(z))=1}}
  L\left(\frac{1}{2},E_d\right) \\
  &
  \leq \bigg(\sum_{\substack{d\in \Omega(a,\epsilon)
  \\(d,P(z))=1}}
  L\left(\frac{1}{2},E_d\right)R^{-}(d)^2
  \Phi\left(\frac{\epsilon d}{X}\right)\bigg) \bigg/
  \bigg(\sum_{\substack{d\in \Omega(a,\epsilon)\\(d,P(z))=1}}
  R^{-}(d)^2
  \Phi\left(\frac{\epsilon d}{X}\right)\bigg)\\
  & \ll \prod_{\substack{\mathcal{L}^2\leq p\leq \exp((\log \mathcal{L})^2)}}
  \left(1+2a(p)b^{-}(p)\tilde{h}(p)p^{-\frac{1}{2}}\right) \\
  & \ll \exp\Big(-2\mathcal{L}\sum_{\mathcal{L}^2\leq p\leq \exp((\log \mathcal{L})^2)}
\frac{a(p)^2}{p\log p}\Big)\\
  & =\exp\left(-\left(2\sqrt{\frac{W-19.73}{22W+12}}
  +o(1)\right)\frac{\sqrt{\log X}}{\sqrt{\log\log X}}\right).
\end{align*}
This completes the proof of Theorem \ref{theorem:extremalvalue}.
\end{proof}

\section*{Acknowledgements}
The authors would like to thank Professors Xiannan Li, Jianya Liu, and Ping Xi for valuable advice.
They want to thank the referees for the constructive comments and suggestions, which definitely improve the readability and quality of the paper.


\end{document}